\documentclass[12pt,twoside]{amsart}
\usepackage{latexsym,amsmath,amsopn,amssymb,amsthm,amsfonts}
\usepackage[T2A]{fontenc}
\usepackage[cp1251]{inputenc}
\usepackage[english]{babel}
\usepackage{graphicx}

\newtheorem{theorem}{Theorem}

\newtheorem{proposition}{Proposition}

\newtheorem{remark}{Remark}

\newtheorem{definition}{Definition}

\newtheorem{corollary}{Corollary}

\newtheorem{lemma}{Lemma}

\newcommand{\Ad}{\operatorname{Ad}}

\newcommand{\sgn}{\operatorname{sgn}}

\newcommand{\const}{\operatorname{const}}

\newcommand{\GL}{\operatorname{GL}}

\newcommand{\ad}{\operatorname{ad}}

\newcommand{\arch}{\operatorname{arch}}

\newcommand{\sh}{\operatorname{sh}}

\newcommand{\ch}{\operatorname{ch}}

\begin{document}
	
\title[PMP, (co)adjoint representation, geodesics]{Pontryagin maximum principle, (co)adjoint representation, and normal geodesics of left-invariant (sub-)Finsler metrics on Lie groups}
\author{V.~N.~Berestovskii, I.~A.~Zubareva}
\address{Sobolev Institute of Mathematics, \newline
Russia, 630090, Novosibirsk, Acad. Koptyug avenue, 4;\newline
Novosibirsk State University,\newline
Russia, 630090, Novosibirsk, Pirogova str., 1}
\email{vberestov@inbox.ru}
\address{Sobolev Institute of Mathematics,\newline
Russia, 644099, Omsk, Pevtsova str., 13}
\email{i\_gribanova@mail.ru}
\begin{abstract}
On the ground of origins of the theory of Lie groups and Lie algebras, their (co)adjoint representations, and the
Pontryagin maximum principle for the time-optimal problem  are given an independent foundation for methods of geodesic vector 
field  to search for normal geodesics of left-invariant (sub-)Finsler metrics on Lie groups and to look for the corresponding locally 
optimal controls in (sub-)\\Riemannian case, as well as some their applications.

\vspace{2mm}

\vspace{2mm}
\end{abstract}
\maketitle

\section*{Introduction}

An extensive geometric research subject  is the class of homogeneous Riemannian manifolds which includes  Lie groups with left-invariant 
Riemannian metrics \cite{BerNik12} and is a part of the  class of homogeneous Finsler manifolds \cite{Deng12}. Every homogeneous 
Riemannian manifold is the image of some Lie group with a left-invariant Riemannian metric relative to a Riemannian submersion.

After  Gromov's 1980s papers,  homogeneous sub-Finsler manifolds, in particular, sub-Riemannian manifolds
were actively studied \cite{BR96}--- \cite{AS}. Their investigation is based on the Rashevsky--Chow theorem which states that 
any two points  of a connected manifold can be joined by a piecewise smooth curve tangent to a given totally
nonholonomic distribution \cite{R38}, \cite{C39}. Аn independent proof of some its version for Lie groups with left-invariant 
sub-Finsler metrics is given in Theorem \ref{contr}.

All homogeneous (sub-)Finsler manifolds are contained in the class of locally compact homogeneous spaces with intrinsic metric.
This class is  a complete metric space with respect to the Busemann-Gromov-Hausdorff metric introduced in \cite{BerGorb14}. 
Its everywhere dense subset is the class of Lie groups with left-invariant Finsler metrics. In addition,

1) each homogeneous locally compact space $M$ with intrinsicr metric is the limit of some sequence of homogeneous manifolds $M_n$  
with intrinsic metrics, bonded by submetries \cite{Ber88}, \cite{Ber891}, \cite{Ber87}, \cite{BerGu00};

2) every homogeneous manifold with intrinsic metric is the quotient  space 
$G/H$ of some connected Lie group $G$ by its compact subgroup $H,$ equipped with
$G$-invariant Finsler or sub-Finsler metric $d$; in particular, it
may be Riemannian or sub-Riemannian metric \cite{Ber88}, \cite{Ber881}, \cite{Ber89}; 

3) moreover, according to a form of metric $d$, there exists a left-invariant Finsler, sub-Finsler, Riemannian or sub-Riemannian 
metric $\rho$ on $G$ such that  the canonical projection $(G,\rho)\rightarrow (G/H,d)$ is a submetry \cite{Ber89}.

The search for geodesics of homogeneous (sub)-Finsler manifolds are reduced to the case of 
Lie groups with left-invariant (sub)-Finsler metrics.

The shortest arcs on Lie groups with left-invariant (sub)-Finsler metrics are optimal
trajectories of the corresponding left-invariant time-optimal problem on Lie groups \cite{Ber88}.
This permits to apply the Pontryagin maximum principle (PMP) for their search \cite{PBGM}.
By this method, in \cite{Ber94} are found all geodesics and
shortest arcs of an arbitrary sub-Finsler metric on the three-dimensional Heisenberg group.

In \cite{Ber14} is proposed a search method of normal geodesics on Lie groups with  left-invariant sub-Riemannian metrics. 
The method is applicable to Lie groups with left-invariant Riemannian metrics, since all
their geodesics are normal.

In this paper, to find geodesics of left-invariant (sub-)Finsler metrics on
Lie groups and corresponding locally optimal controls in (sub-)Riemannian case we use  the geodesic vector field method 
(Theorems \ref{main},\ref{vf}) and an improved version of method from \cite{Ber14}, applying (co)adjoint representations.
The version is based on differential equations  from Theorem \ref{kok} for controls, using only the structure constants of Lie algebras
of Lie groups.

An interesting feature of these two methods in (sub-)Riemannian case  is that  geodesics vector fields on Lie groups (their integral 
curves are geodesics, i.e., locally optimal trajectories) and locally optimal controls on Lie algebras of these 
Lie groups can be determined independently of each other, although there is a connection between them. Moreover, controls on different Lie algebras
could be solutions of the same mathematical pendulum equation   (see sections \ref{so3}--\ref{se2}).

Analogues of Theorems \ref{hameq1} and \ref{main} (but for the last theorem is only along one geodesic) are proved in the book \cite{Jur97} 
on the basis of more complicated concepts and apparatus. Apparently, other researchers did not apply PMP {\it for the time-optimal problem}
to find geodesics of left-invariant metrics on Lie groups.

\section{Preliminaries}

A smooth manifold  $G$ which is a group with respect to an operation $\cdot$ is called the Lie group if the operations of multiplication and inversing 
are smooth maps. Smooth map of Lie groups that is a homomorphism is called a homomorphism of Lie groups. Monomorphisms, epimorphisms, 
and isomorphisms of Lie groups are defined in a similar way. A subgroup $H$ of a Lie group $G$ which is its smooth submanifold is called the Lie 
subgroup of the Lie group $G$. By E.Cartan's theorem, every closed subset $H$ of the Lie group $G$, which is its subgroup, is the Lie subgroup 
of the Lie group $G$ \cite{Ad}.

The concept of the virtual Lie subgroup of a Lie group generalizes the concept of the Lie subgroup of a Lie group.
A subgroup $H$ of a Lie group $G$ is called its virtual Lie subgroup, if $H$ admits the structure of the Lie group such that its  topology base
consists of connected components of open subsets of the induced topology and the inclusion map of $H$ in $G$ is an (injective) homomorphism of Lie groups.

The left and the right shifts $l_g: h\in G\rightarrow g\cdot h,$ $ r_g: h\in G\rightarrow h\cdot g,$ $g,h\in G,$
of the Lie group $(G,\cdot)$ by an element $g$ are diffeomorphisms with the inverse shifts  $l_{g^{-1}},$ 
$r_{g^{-1}},$ and their differentials $(dl_g)_h: T_hG\rightarrow T_{gh}G$ (respectively,  $(dr_g)_h: T_hG\rightarrow T_{hg}G)$ are linear 
isomorphisms of tangent vector spaces to $G$ at corresponding points.

A (smooth) vector field $V: G\rightarrow TG,\,\,V: g\in G\rightarrow T_gG$ on the Lie group $G$ such that $V\circ l_h = d(l_h)\circ V$ for all $h\in G,$ is called 
the left-invariant vector field on $G$. The right-invariant vector field on $G$ is defined in a similar way. 
Every left-invariant vector field on the Lie group $G$ has a form
\begin{equation}
\label{V}
V(g)=(dl_g)_e(v),\quad v\in T_eG,
\end{equation}
where $e$ is the unit of the group $G$.

A homomorphism of Lie groups 
$\phi: (\mathbb{R},+)\rightarrow (G,\cdot)$ is called the $1$--parameter subgroup of the Lie group
$(G,\cdot)$. \label{onep}
Every $1$--parameter subgroup $\phi(t), t\in \mathbb{R},$ of a Lie group $G$ is an integral curve of a left-invariant vector field $V$ on $G$ 
with formula (\ref{V}), where $v= (d\phi)_0(\overline{e}),$ and $\overline{e}\in T_0\mathbb{R}$ is the vector with the component 1.

For a vector $v\in T_eG,$ we denote by  $V_v$ and  $\phi_v$ respectively the left-invariant vector field $V$ on $G,$ 
defined by (\ref{V}), and the $1$--parameter subgroup $\phi=\phi(t),$ $t\in\mathbb{R},$ in $G$ with condition $(d\phi)_0(\overline{e})=v$.
The exponential map $\exp=\exp_G: T_eG\rightarrow G$ is defined by formula $v\in T_eG\rightarrow \phi_v(1).$
If $f: G\rightarrow H$ is a homomorphism of Lie groups then
\begin{equation}
\label{expc}
f\circ \exp_G= \exp_H\circ (df)_e.
\end{equation}
For each vector $v\in T_eG,$ we have $(d\exp)_{0}(v)=v,$ where $0$ is zero of the tangent vector space $T_eG.$
As a result, there exist open neighborhoods $U$ of zero in  $T_eG$ and $W$ of unit $e$ in $G$ such that $\exp: U\rightarrow V$ is a diffeomorphism.
If $\dim(G)=n$ then after introduction of arbitrary Cartesian coordinates $(x_1,\dots, x_n)$ with zero origin $0$ in the tangent vector space 
$T_eG,$ it is naturally identified with $\mathbb{R}^n.$ Then $\exp^{-1}: V\rightarrow U\subset \mathbb{R}^n$ is a local chart (a coordinate system) 
on $G$ in the neighborhood $V$ of the point $e\in G.$ This coordinate system in $V$ is called {\it a coordinate system of the first kind}.
A family of local charts 
$\exp^{-1}\circ l_{g^{-1}}: g\cdot V\rightarrow U\subset \mathbb{R}^n,$ $g\in G,$
sets a smooth structure on  $G,$ identical with the initial smooth structure of the Lie group.

The group $\GL(n)=\GL(n,\mathbb{R})$ of all nondegenerate real squared $(n\times n)$-matrices is a Lie group  
relative to the global map that associates to each matrix $g\in \GL(n)$ its elements $g_{ij},$ $i,j=1,\dots n.$

Obviously, for every $g\in G$ the mapping $I(g): G\rightarrow G$ such that 
$$I(g)(h)= g\cdot h\cdot g^{-1}=(l_g \circ r_{g^{-1}})(h)=  (r_{g^{-1}} \circ l_{g})(h)$$
is an automorphism of the Lie group $(G,\cdot),$ $I(g)(e)=e,$ and the differential
$$(dI(g))_e: = dl_g\circ dr_{g^{-1}}: T_eG\rightarrow T_eG$$ is a nondegenerate linear map
(i.e. an element of the Lie group $\GL(n)$ relative to
some vector basis in $T_eG$, if $\dim G=n$), denoted with $\Ad(g)$. 
The calculation rule for the  differential  of composition gives
$$\Ad(g_1\cdot g_2)= (dI(g_1\cdot g_2))_e= (d(I(g_1)\circ I(g_2)))_e=   (dI(g_1))_e\circ (dI(g_2))_e= \Ad(g_1)\circ \Ad(g_2),$$
i.e., $\Ad: G\rightarrow \GL(n)$ is a homomorphism of Lie groups, called the adjoint representation of the Lie group $G$. 
By formula (\ref{expc}),
\begin{equation} 
\label{igexp}
I(g)\circ \exp= \exp\circ \Ad(g), g\in G,
\end{equation}
the kernel of the homomorphism $\Ad$ for a connected Lie group $G$ is the center of the Lie group $G,$
\begin{equation}
\label{digexp}
\Ad\circ \exp_G= \exp_{\GL(n)}\circ (d\Ad)_e.
\end{equation}

Set $\mathfrak{g}:=T_eG$ for a Lie group $(G,\cdot),$  $\mathfrak{gl}(n):= T_E\GL(n)= M(n)$ 
for the Lie group $\GL(n),$ where $M(n)$ is the vector space of all real $(n\times n)$-matrices, 
$\ad =\ad_{\mathfrak{g}}:= (d\Ad)_e;$  $L(X,Y)$ is the (real) vector space of linear maps from the real vector space  $X$ to 
the real vector space $Y$; $B(X\times Y, Z)$ is the vector space of bilinear maps from $X\times Y$ to $Z$. It is clear that
$$\ad \in L(\mathfrak{g}, L(\mathfrak{g},\mathfrak{g}))=B(\mathfrak{g}\times \mathfrak{g},\mathfrak{g}).$$
{\it A vector $[v,w]: =\ad(v)(w)\in \mathfrak{g},$ $v,w\in \mathfrak{g}$, is called the 
Lie bracket of vectors $v,w\in \mathfrak{g}.$ The pair $(\mathfrak{g},[\cdot,\cdot])$ is called the Lie algebra of the Lie group $(G,\cdot)$}.  
The definition implies that the Lie bracket operation is bilinear. It is clear that
$$\frac{\partial}{\partial s}[\exp(tv)\exp(sw)\exp(-tv)](0)=\Ad(\exp(tv))(w),$$
\begin{equation}
\label{bra}
[v,w]= \frac{\partial}{\partial t}\left(\frac{\partial}{\partial s}[\exp(tv)\exp(sw)\exp(-tv)](0)\right)(0),
\end{equation}

The formula (\ref{bra}) and the bilinearity of the Lie bracket imply the skew symmetry of the Lie bracket and the triviality of the Lie algebra 
of any commutative Lie group; for a connected Lie group the converse statement is also true. It follows from formulae  (\ref{expc}), 
(\ref{bra}) that if $f: G\rightarrow H$ is a homomorphism of Lie groups and  $(\mathfrak{h},[\cdot,\cdot])$ is the Lie algebra of the 
Lie group $H$, then for any elements $v,w\in \mathfrak{g},$  
$$(df)_e([v,w])=[(df)_e(v),(df)_e(w)].$$ 
In other words, the differential $(df)_e: \mathfrak{g}\rightarrow \mathfrak{h}$ is a homomorphism of Lie algebras 
$(\mathfrak{g},[\cdot,\cdot])$ and $(\mathfrak{h},[\cdot,\cdot])$ of Lie groups $G$ and $H.$  As a corollary, Lie algebras of locally 
isomorphic Lie groups are isomorphic (the converse statement is also true) and
\begin{equation}
\label{adgi}
\Ad(g)([v,w])= [\Ad(g)(v),\Ad(g)(w)],\quad g\in G,\quad v,w\in \mathfrak{g}.
\end{equation}
The substitution $g=\exp(tu),$ $u\in \mathfrak{g},$ to this formula and the differentiation by $t$ at $t=0$ gives 
the following formula
\begin{equation}
\label{difalg}
[u,[v,w]]= [[u,v],w]+ [v,[u,w]],\quad u,v,w\in (\mathfrak{g},[\cdot,\cdot]),
\end{equation}
which is equivalent by the skew symmetry of the Lie bracket to {\it the Jacobi identity} 
\begin{equation}
\label{jak}
[ u,[v,w]]+ [v,[w,u]]+ [w,[u,v]]=0. 
\end{equation}

It is well-known that
\begin{equation}
\label{expgl}
\exp_{\GL(n)}(A)= \exp A = \sum_{k=0}^{\infty}\frac{A^k}{k!},\quad A\in\mathfrak{gl}(n),
\end{equation}
which together with  (\ref{bra}) imply
\begin{equation}
\label{slgl}
[A,B]= AB - BA,\quad A,B\in (\mathfrak{gl}(n),[\cdot,\cdot]).
\end{equation}

\section{Theoretic results}

\begin{definition}
	\label{gen}
	Let $(\mathfrak{l},[\cdot,\cdot])$ be a Lie algebra; $\mathfrak{p}, \mathfrak{q}\subset \mathfrak{l}$ are nonzero vector subspaces. By definition, 
	$$[\mathfrak{p},\mathfrak{q}]= \{[v,w]: v\in \mathfrak{p}, w\in \mathfrak{q}\}.$$
	If $\dim(\mathfrak{p})\geq 2$ then by definition,
	$$\quad\mathfrak{p}^1=\mathfrak{p},\quad \mathfrak{p}^{k+1}=[\mathfrak{p},\mathfrak{p}^k],\quad \mathfrak{p}_m= \sum_{k=0}^{m}\mathfrak{p}^k.$$ 
	The vector subspace $\mathfrak{p}\subset \mathfrak{l}$ generates the Lie algebra  $(\mathfrak{l},[\cdot,\cdot])$, if 	$\mathfrak{l}=\mathfrak{p}_m$ for some natural number $m;$ the smallest number $m:=s$ with such property is called the generation degree 
	(of the algebra $(\mathfrak{l},[\cdot,\cdot])$ by the subspace $\mathfrak{p}$).  
\end{definition}
It is clear that subsets from Definition \ref{gen} are vector subspaces of $\mathfrak{l}.$ 
\begin{definition}
	\label{adapt}
	Let us assume that the vector subspace $\mathfrak{p}\subset \mathfrak{l}$ generates the Lie algebra  $(\mathfrak{l},[\cdot,\cdot]),$
	$2\leq \dim(\mathfrak{p})< \dim(\mathfrak{l}),$ $s$ is the generation degree, $r_m,$ $m=1,\dots, s,$ are dimensions (ranks) of  the 
	spaces $\mathfrak{p}_m.$ Thus $2\leq r_1< r_2<\dots < r_s,$ $r_1=\dim(\mathfrak{p})=r,$ $r_s=\dim(\mathfrak{l})=n$. 
	A basis $\{e_1,\dots, e_{r_s}\}$ of the Lie algebra $\mathfrak{l}$ is called adapted to the subspace
	$\mathfrak{p},$ if $\{e_1,\dots, e_{r_m}\}$ is a basis of the subspace $\mathfrak{p}_m$ for every  $m=1,\dots, s$.    
\end{definition}

Let $\{e_1,\dots,e_r\}$ be any basis of the vector subspace $\mathfrak{p}\subset \mathfrak{g},$ generating the Lie algebra  
$(\mathfrak{g},[\cdot,\cdot])$ of a Lie group $(G,\cdot).$

\begin{theorem}
	\label{contr}
	Let $(G,\cdot)$ be a connected Lie group and a vector subspace $\mathfrak{p}\subset \mathfrak{g}$ generates Lie algebra 
	$(\mathfrak{g},[\cdot,\cdot]).$ Then the  control system 
	\begin{equation}
	\label{dyn}
	\dot{g}=(dl_g)(u),\quad u\in \mathfrak{p},
	\end{equation}
	is controllable (attainable) by means of piecewise constant controls
	\begin{equation}
	\label{co}
	u=u(t)\in \mathfrak{p},\quad 0\leq t\leq T,
	\end{equation}
	where $u(t)=\pm e_j,$ $j=1,\dots, r,$ in the constancy segments of the control.
	In other words, for any elements $g_0,g_1\in G$ there exists  a piecewise constant control  (\ref{co})  of this type such that $g(T)=g_1$ 
	for solution of the Cauchy problem
	$$\dot{g}(t)=dl_{g(t)}(u(t)), \quad g(0)=g_0.$$ 
\end{theorem}

\begin{proof}
	We shall apply the notation from Definitions \ref{gen} and \ref{adapt}. 
	
	Let usl construct an adapted basis $\{e_1,\dots, e_n\}$ to the subspace $\mathfrak{p}$  of the Lie algebra
	$(\mathfrak{g},[\cdot,\cdot])$ by induction on $m=1,\dots, s.$ 
	
	$m=1.$    First $r$ vectors of the basis coincide with vectors of basis for the space 
	$\mathfrak{p}^1=\mathfrak{p}$ chosen before Theorem \ref{contr}.

	$m=2.$  It is clear that we can take some vectors of a form $e_j=[e_{i_j},e_{k_j}]\in \mathfrak{p}^2,$ $j=r+1,\dots, r_2,$ where $i_j,$ $k_j$ 
	are some of numbers $1,\dots,r.$
	
	Let us assume that vectors  $e_1,\dots, e_{r_m}$ are constructed, where $2\leq m<s.$ Then we can take some vectors of a form
	$e_j=[e_{i_j},e_{k_j}]\in \mathfrak{p}^{m+1},$ $j=r_m+1,\dots, r_{m+1},$ where
	$i_j$ (respectively, $k_j$) are some of numbers $1,\dots, r$ (respectively, $r_{m-1}+1,\dots, r_m$).
	
	As a result, each vector $e_j,$ where $r_{m-1}< j\leq r_m,$ $m=2,\dots s,$ has a form
	\begin{equation}
	\label{vb}
	e_j=[e_{i_m(j)},[\dots, [e_{i_2(j)},e_{i_1(j)}]\dots ]],\quad 1\leq i_l(j)\leq r,\quad l=1,\dots, m.
	\end{equation}
	
	We claim that if every such vector $e_j$ is replaced by a vector  $e'_j$ of a form
	\begin{equation}
	\label{nvb}
	e'_j= (Ad(\exp(t_m e_{i_m(j)})\circ \dots \circ Ad(\exp(t_2 e_{i_2(j)})))(e_{i_1(j)})
	\end{equation}
	with sufficiently small nonzero numbers $t_2,\dots, t_m$ (preserving vectors $e_1,\dots, e_r$), then  we get again some basis in 
	$\mathfrak{g}$ (not necessarily adapted to the subspace $\mathfrak{p}$). 
	
	Indeed, on the basis of formulae (\ref{nvb}), (\ref{digexp}),
	$$e'_j= (\exp(t_m\ad(e_{i_m(j)}))\circ \dots \circ \exp(t_2\ad(e_{i_2(j)})))(e_{i_1(j)})= $$
	$$((E+t_m\ad(e_{i_m(j)})+O(t_m^2))\circ\dots \circ (E+t_2\ad(e_{i_2(j)})+O(t_m^2)))(e_{i_1(j)})=$$
	$$e_{i_1(j)}+ t_2[e_{i_2(j)},e_{i_1(j)}]+ \dots +(t_m\dots t_2)[e_{i_m(j)},[\dots,[e_{i_2(j)},e_{i_1(j)}]\dots]]+\sum_{k=2}^m o(t_k).$$
	We see from here and (\ref{vb}) that removing the last sum, we get a vector from
    $\mathfrak{p}_m$ that is equal to the vector $(t_m\dots t_2)e_j$ up to the module of the subspace $\mathfrak{p}_{m-1}$.
    This implies the statement from the previous paragraph.
	
	For simplicity, later on each such vector $e'_j$ is denoted by $e_j.$ 
	
	On the groud of formulae (\ref{nvb}) and (\ref{igexp}), 
	\begin{equation}
	\label{ige}
	\exp{(se_j)}=(I(\exp(t_m e_{i_m(j)}))\circ\dots\circ I(\exp(t_2 e_{i_2(j)})))(se_{i_1(j)}),\quad s\in \mathbb{R}.
	\end{equation}  
	
	Let us show that the statement of Theorem \ref{contr} is true for elements $g_0=e$ and $g_1=\exp(se_j).$  
	For this, we apply a control
	$$u=u(\tau),\quad 0\leq \tau\leq |s|+ 2\sum_{k=2}^{m}|t_k|,$$ 
    where 
	$$u(\tau)=\sgn(t_l) e_{i_l(j)}, \quad \sum_{k=l}^m|t_k|-|t_l|\leq \tau\leq \sum_{k=l}^m|t_k|,\quad l=2,\dots m,$$
	$$u(\tau)=\sgn(s) e_{i_1(j)}, \quad \sum_{k=2}^m|t_k|\leq \tau\leq \sum_{k=2}^m|t_k|+ |s|,$$
	$$u(\tau)=-\sgn(t_l) e_{i_l(j)},\quad  \sum_{k=2}^m|t_k|+ |s|+ \sum_{k=2}^l|t_k|-|t_l|\leq \tau\leq \sum_{k=2}^m|t_k|+ |s|+ \sum_{k=2}^l|t_k|,$$ 
	where $l=2,\dots, m.$
	Then it follows from the definition of $I(g),$ $g\in G,$ and the equation (\ref{ige}) that solution of the  Cauchy problem 
	for the system (\ref{dyn}) with $g(0)=e$ and with given control
	$u=u(\tau)$ is a piecewise smooth curve
	$$g(\tau)=\exp\left(\left(\tau- \sum_{k=l}^m|t_k|+|t_l|\right)\sgn(t_l) e_{i_l(j)}\right), \quad \sum_{k=l}^m|t_k|-|t_l|\leq \tau\leq \sum_{k=l}^m|t_k|;$$
	$$g(\tau)=\exp\left(\left(\tau - \sum_{k=2}^m|t_k|\right)\sgn(s) e_{i_1(j)}\right), \quad \sum_{k=2}^m|t_k|\leq \tau\leq \sum_{k=2}^m|t_k|+ |s|;$$
	$$g(\tau)=\exp\left(-\left(\tau - \left(\sum_{k=2}^m|t_k|+ |s|+ \sum_{k=2}^l|t_k|-|t_l|\right)\right)\sgn(t_l) e_{i_l(j)}\right), $$ 
	$$\sum_{k=2}^m|t_k|+ |s|+ \sum_{k=2}^l|t_k|-|t_l|\leq \tau\leq \sum_{k=2}^m|t_k|+ |s|+ \sum_{k=2}^l|t_k|,$$
	where $ l=2,\dots, m.$ In addition, $g\left(|s|+ 2\sum_{k=2}^{m}|t_k|\right)=\exp(se_j).$

	It follows from proved assertions that  for any collection $(s_1,\dots, s_n)\in \mathbb{R}^n$ the statement of Theorem \ref{contr} holds for elements
	$$g_0=e,\quad g_1=\Phi(s_1,\dots, s_n):= \exp(s_1e_1)\dots \exp(s_ne_n).$$ 
	In addition,
	$$\frac{\partial \Phi}{\partial s_i}(0,\dots, 0)=e_i,\quad t=1,\dots, n.$$
	Then on the ground of the inverse mapping theorem the map $\Phi$ is a diffeomorphism of
	some open neighborhood $W$ of zero $(0,\dots,0)$ in $\mathbb{R}^n$ onto some open neighborhood $V$ 
	of the unit $e$ in $G.$
		
	It follows from previously proved assertions that the statement of Theorem \ref{contr} holds for $g_0=e$ and any element $g_1\in V^k$, where  $k$ is 
	arbitrary natural number, hence for any element $g_1\in W:=\cup_{k=1}^{\infty}V^k.$ 
	This set is nonempty, open and closed in $G.$ First two properties are obvious; we shall prove that the set is closed. Set 
	$$V_0:= V\cap V^{-1},\quad\mbox{where}\quad V^{-1}=\{g^{-1}: g\in V\}.$$
	It is clear that $V_0$ is a symmetric neighborhood of the unit $e$ in $G,$ i.e., $V_0^{-1}=V_0.$
	Let $g_1\in \overline{W},$ where $ \overline{W}$ is the closure of $W.$ Then $g_1V_0\cap W\neq \emptyset,$ consequently, $g_1V_0\cap V^k\neq \emptyset$  for some
	$k,$ so there exists $g\in g_1V_0\cap V^k,$ $g=g_1v_0$  for $v_0\in V_0.$ Then
	$$g_1=gv_0^{-1}\in gV_0\subset gV\subset V^{k}V=V^{k+1}\subset W.$$
	Therefore $W$ is an open and closed set and $W=G,$ because $G$ is connected. 
	
    Now if $g_0, g_1\in G$ then $g_0=l_{g_0}(e),$ $g_1=l_{g_0}((g_0)^{-1}g_1),$ and 
    since the statement of Theorem \ref{contr} holds for elements $e$ and $(g_0)^{-1}g_1,$
     then it holds for $g_0$ and $g_1.$
\end{proof}

It follows from the proof of Theorem \ref{contr} that the triple $(V,\Phi^{-1},W)$ is a local chart in $G.$ 
The corresponding coordinate system is called {\it the coordinate system of the second kind}.

Every left-invariant (sub-)Finsler metric $d=d_F$ on a connected Lie group $G$ with Lie algebra $(\mathfrak{g},[\cdot,\cdot])$ 
is defined by a subspace  $\mathfrak{p}\subset \mathfrak{g}$, generating  $\mathfrak{g}$, and some norm $F$ on $\mathfrak{p}.$ 
A distance $d(g,h)$ for $g,h\in G$ is defined as the infimum of lengths $\int_0^T|\dot{g}(t)|dt$ of piecewise smooth paths 
$g=g(t),$ $0\leq t\leq T,$ such that $dl_{g(t)^{-1}}\dot{g}(t)\in \mathfrak{p}$ and $g(0)=g,$ $g(T)=h;$ $T$ is not fixed, 
$|\dot{g}(t)|=F(dl_{g(t)^{-1}}\dot{g}(t)).$ 
The existence of such paths and, consequently, the finiteness of $d$ are guaranteed by Theorem \ref{contr}.
Obviously, all three metric properties for $d$ are fulfilled.
If $\mathfrak{p}=\mathfrak{g}$ then $d$ is a left-invariant Finsler metric on $G$; if
$F(v)=\sqrt{\langle v, v\rangle},$ $v\in \mathfrak{p},$ where $\langle\cdot,\cdot\rangle$ is some scalar product on 
$\mathfrak{p},$ then  $d$ is a left-invariant sub-Riemannian metric on $G,$ and $d$ is a left-invariant 
Riemannian metric, if additionally $\mathfrak{p}=\mathfrak{g}.$

The following statements were proved in \cite{Ber881}. The space $(G,d)$ is a locally compact and complete. Then 
in consequence of S.E.~Con--Vossen theorem the space $(G,d)$ is a geodesic space, i.e. for any elements $g,h\in G$ 
there exists a shortest arc $c=c(t),$ $0\leq t\leq T,$ in $(G,d),$ which joins them. This means that $c$ is a continuous curve in $G,$ 
whose length in the metric space $(G,d)$ is equal to $d(g,h).$ Therefore we can assume that $c$ is parameterized by arc 
length, i.e. $T=d(g,h)$ and $d(c(t_1),c(t_2))=t_2-t_1$ if $0\leq t_1\leq t_2\leq d(g,h).$ Then $c=c(t),$ $0\leq t\leq d(g,h),$ is a Lipschitz 
curve relative to the smooth structure of the Lie group $G$. Therefore this curve is absolutely continuous. Then in consequence of  well--known 
theorem from mathematical analysis, there exists  a measurable, almost everywhere defined derivative function
$\dot{c}(t),$ $0\leq t\leq d(g,h)$, and $c(t)=c(0)+ \int_0^t\dot{c}(\tau)d\tau,$ $0\leq t\leq T.$

\begin{theorem}
	\label{topt}\cite{Ber88}
	Every shortest arc $g=g(t),$ $0\leq t\leq T=d(g_0,g_1)$,  in $(G,d)$ with $g(0)=g_0,$ $g(T)=g_1,$ is a solution of the time-optimal 
	problem for the control system (\ref{dyn}) with compact control region
		$$U=\{u\in \mathfrak{p}: F(u)\leq 1\}$$
	and indicated endpoints.
\end{theorem}

In consequence of Theorem \ref{topt}, one can apply the Pontryagin maximum principle \cite{PBGM} for the time-optimal problem  from Theorem \ref{topt} 
and a covector function $\psi=\psi(t)\in T^{\ast}_{g(t)}$ to find shortest arcs on the Lie group $G$ with left-invariant sub-Finsler metric $d.$ 
The function $\psi$ can be considered as a left-invariant $1$-form on $(G,\cdot)$ and therefore it is natural to identify it with a covector function $\psi(t)\in\mathfrak{g}^{\ast}=T_e^{\ast}G.$ Then every optimal
trajectory $g(t),$ $0\leq t\leq T,$ is determined by some (piecewise continuous) optimal control $\overline{u}=\overline{u}(t)\in U,$ $0\leq t\leq T.$ 
Moreover, for some non-vanishing absolutely continuous function $\psi=\psi(t),$ $0\leq t\leq T,$ we have
\begin{equation}
\label{H} 
H=H(g,\psi,u)=\psi((dl_g)(u))=\psi(u),
\end{equation}
\begin{equation}
\label{partial}
\dot{g}=\frac{\partial H}{\partial \psi},\quad \dot{\psi}=-\frac{\partial H}{\partial g},
\end{equation}
\begin{equation}
\label{max}
H(\tau):=H(\psi(\tau),\overline{u}(\tau))=\psi(\tau)(\overline{u}(\tau))=\max_{u\in U}\psi(\tau)(u)
\end{equation}
at continuity points $\tau$  of the optimal control $\overline{u}=\overline{u}(t)$. 

\begin{definition}
	Later on,  an extremal for the problem from Theorem \ref{topt} is called a parametrized curve $g = g(t)$, $t\in \mathbb{R},$ 
	satisfying PMP for the time-optimal  problem.
	\end{definition}

\begin{remark}
	\label{pos}
For every extremal, $H(t)=\const:= M_0\geq 0,$ $t\in \mathbb{R},$ \cite{AS, PBGM}.  
\end{remark}

\begin{definition}
	An extremal is called normal (abnormal), if $M_0 > 0$ ($M_0=0$). Every normal extremal  is
	parameterized by arc length; proportionally changing $\psi=\psi(t),$ $t\in \mathbb{R},$ if it is necessary, one can assume that $M_0=1.$ 
	Every normal extremal for a left-invariant (sub-)Riemannian metric
	on a Lie group is a geodesic, i.e. a locally shortest curve \cite{LS95}.
	\end{definition}

\begin{theorem}
	\label{hameq}\cite{Ber14}
	The Hamiltonian system for the function $H$ on the Lie group $G=\GL(n)$ with the Lie algebra
	$\mathfrak{g}=\mathfrak{gl}(n)$ has a form
	\begin{equation}
	\label{sgo}
	g^{\prime}=g\cdot u,\quad g\in G, \quad u\in \frak{g},
	\end{equation}
	\begin{equation}
	\label{hame0}
	\psi(v)^{\prime}= \psi([u,v]),\quad g\in G, \quad u,v\in \frak{g}.
	\end{equation}
\end{theorem}

\begin{proof}
	Each element $g\in G\subset \GL(n)\subset \mathbb{R}^{n^2}$ is defined by  its standard matrix coordinates  $g_{ij},$ 
	$i,j=1,\dots n,$ and $\psi$ is defined  by its components 	 $\psi_{ij}=\psi(e_{ij}),$ $i,j=1,\dots,n,$ where 
	$e_{ij}\in \mathfrak{g}$ is a matrix having $1$ in the $i$th row and the $j$th column and 0 in all other places.
	 	
	In consequence of (\ref{H}),
	\begin{equation}
	\label{H1}
	H(\psi,g,u)=\sum_{i,j=1}^n\psi_{ij}\left(\sum_{l=1}^ng_{il}u_{lj}\right)=
	\sum_{l,j=1}^n(g^T\psi)_{lj}u_{lj}.
	\end{equation}
	
	The variables $g_{ij},$ $\psi_{ij}$ must satisfy the Hamiltonian system of equations
	\begin{equation}
	\label{dx}
	g_{ij}^{\prime}=\frac{\partial H}{\partial \psi_{ij}}(\psi,g,u)=\sum_{l=1}^ng_{il}u_{lj}=(gu)_{ij},
	\end{equation}
	\begin{equation}
	\label{dp}
	\psi_{ij}^{\prime}= -\frac{\partial H}{\partial g_{ij}}=-\sum_{m=1}^{n}\psi_{im}u_{jm}=-(\psi u^T)_{ij}.
	\end{equation}
	The formula (\ref{dx}) is a special case of the formula (\ref{sgo}). It is clear that
	$$\psi(v)=\psi(gv)=\sum_{i,j=1}^n\psi_{ij}\left(\sum_{l=1}^ng_{il}v_{lj}\right).$$
	On the ground of formulae (\ref{dx}) and (\ref{dp}) we get from here that
	$$(\psi(v))^{\prime}=\sum_{i,j=1}^n \psi_{ij}^{\prime}\left(\sum_{l=1}^ng_{il}v_{lj}\right)+\sum_{i,j=1}^n\psi_{ij}\left(\sum_{l=1}^n g_{il}^{\prime}v_{lj}\right)=$$
	$$-\sum_{i,j=1}^n\left(\sum_{m=1}^{n}\psi_{im}u_{jm}\sum_{l=1}^ng_{il}v_{lj}\right)+\sum_{i,j=1}^n\psi_{ij}\left(\sum_{l,m=1}^ng_{im}u_{ml}v_{lj}\right)=$$
	$$-\sum_{i,j=1}^n\psi_{ij}\left(\sum_{l=1}^ng_{il}(vu)_{lj}\right)+\sum_{i,j=1}^n\psi_{ij}\left(\sum_{l=1}^ng_{il}(uv)_{lj}\right)=
	\sum_{i,j=1}^n\psi_{ij}(g[u,v])_{ij}=\psi([u,v]),$$
	which proves the formula (\ref{hame0}).
\end{proof}

\begin{theorem}
	\label{hameq1}\cite{Ber14}
	The Hamiltonian system for the function $H$ on a Lie group $G$ with Lie algebra $\mathfrak{g}$ has a form
	\begin{equation}
	\label{sg}
	\dot{g}=dl_{g}(u),\quad g\in G, \quad u\in \frak{g},
	\end{equation}
	\begin{equation}
	\label{hame}
	\psi(v)^{\prime}= \psi([u,v]),\quad g\in G, \quad u,v\in \frak{g}.
	\end{equation}
\end{theorem}

\begin{proof}
	In consequence of Theorem \ref{hameq}, Theorem \ref{hameq1} holds for every matrix Lie group and for every Lie group $(G,\cdot),$ because 
	it is known that $(G,\cdot)$ is  locally isomorphic to some connected Lie subgroup (may be, virtual)  of the Lie group $\GL(n)\subset \mathbb{R}^{n^2}.$  
\end{proof}

It follows from Theorem \ref{hameq1}, especially from (\ref{hame}), and Remark \ref{pos} that

\begin{theorem}
	\label{norm}
	If $\dim(G)=3$, $\dim(\mathfrak{p})\geq 2$  in Theorem \ref{topt} then every extremal of the problem from Theorem  \ref{topt} is normal.
\end{theorem}

The following lemma holds.

\begin{lemma}
	\label{lem2}\cite{Ber17}
	Let $g=g(t)$, $t\in (a,b)$, be a smooth path in the Lie group $G$. Then
	\begin{equation}
	\label{pro}
	(g(t)^{-1})^{\prime}=-g(t)^{-1} g^{\prime}(t) g(t)^{-1}.
	\end{equation}
\end{lemma}

\begin{proof}
	Differentiating the identity $g(t) g(t)^{-1}=e$ by $t$, we get
	$$0=(g(t)g(t)^{-1})^{\prime}=g^{\prime}(t) g(t)^{-1}+g(t)(g(t)^{-1})^{\prime},$$
	whence the equality  (\ref{pro}) follows immediately.
\end{proof}

\begin{theorem}
	\label{dad}\cite{Ber17}
	Let $\psi\in\mathfrak{g}^{\ast}=T^{\ast}_eG$ be a covector, 
	$$\Ad^{\ast}\psi(g):=(\Ad g)^{\ast}(\psi)=\psi\circ Ad(g), \quad g\in G,$$ an action of the coadjoint representation of the Lie group $G$ on $\psi$. Then
	$$(d(\Ad^{\ast}\psi)(w))(v)=((\Ad g_0)^{\ast}(\psi))([u,v]),$$
	if
	$$u,v\in\frak{g},\quad w=dl_{g_0}(u)\in T_{g_0}G,\quad g_0\in G.$$
\end{theorem}

\begin{proof}
	In the case of a matrix Lie group, 
	$$\Ad(g)(v)=gvg^{-1},\quad dl_g(u)=gu,\quad u,v\in\mathfrak{g}, \quad g\in G.$$
	We choose a smooth path $g=g(t)$, $t\in (-\varepsilon,\varepsilon)$, in the Lie group $G$ such that $g(0)=g_0$, $g^{\prime}(0)=w$. 
	Then by Lemma \ref{lem2},	
	$$(d(\Ad^{\ast}\psi)(w))(v)=(\psi(g(t)vg(t)^{-1}))^{\prime}(0)=
	\psi((g(t) vg(t)^{-1})^{\prime}(0))=$$
	$$\psi(g^{\prime}(0)vg_0^{-1}+g_0 v(g(t)^{-1})^{\prime}(0))=
	\psi(g_0uvg_0^{-1}-g_0v(g_0^{-1}g^{\prime}(0)g_0^{-1}))=$$
	$$\psi(g_0uvg_0^{-1}-g_0v(g_0^{-1}g_0ug_0^{-1}))=
	\psi(g_0uvg_0^{-1}-g_0vug_0^{-1})=$$
	$$\psi(g_0[u,v]g_0^{-1})=((\Ad g_0)^{\ast}(\psi))([u,v]),$$
	as required.
\end{proof}

It follows from Theorems \ref{hameq1} and \ref{dad} that 

\begin{theorem}
	\label{main}
	\cite{BerGich}
	
	1. Any normal extremal  $g=g(t):\,\mathbb{R}\rightarrow G$  (parameterized by arc length and with origin $e\in G$), 
	of left-invariant (sub-)Finsler metric $d$ on a Lie group $G$, defined by a norm  $F$ on the subspace 
	$\mathfrak{p}\subset\frak{g}$ with closed unit ball $U$, is a Lipschitz integral curve of the following
	vector field 
	$$v(g)=dl_g(u(g)),\quad u(g)=\psi_0(\Ad(g)(w(g)))w(g),\quad w(g)\in U, $$
	$$\psi_0(\Ad(g)(w(g)))=\max_{w\in U}\psi_0(\Ad(g)(w)),$$
	where $\psi_0\in\frak{g}^{\ast}$ is some fixed covector with $\max_{v\in U}\psi_0(v)=1.$
	
	2. (Conservation law) In addition, $\psi(t)(g(t)^{-1}g^{\prime}(t))\equiv1$ for all $t\in \mathbb{R}$, where $\psi(t):=(\Ad g(t))^{\ast}(\psi_0)$.
\end{theorem}

\begin{remark}
	Every  extremal with origin $g_0$ is obtained by the left shift $l_{g_0}$ from some extremal with origin $e.$ 
\end{remark}

\begin{remark}
	In (sub-)Riemannian case, the vector $u(g)$ is characterized by condition $\langle u(g),v\rangle=\psi_0(Ad(g)(v))$ for all $v\in \mathfrak{p}.$
		In Riemannian case, every  extremal is a normal geodesic, and  we can assume that $\psi_0$ is an 
		unit vector in $(\mathfrak{p}=\mathfrak{g},(\cdot,\cdot)),$ setting $\psi_0(v)=(\psi_0,v),$ $v\in \mathfrak{g}.$ Moreover, $\dot{g}(0)=\psi_0.$
\end{remark}

\begin{corollary}
	\label{bi}
	Every geodesic of a biinvariant Riemannian metric on a Lie group with the unit origin is its $1$-parameter subgroup.
\end{corollary}

\begin{proof} 
	This statement is a consequence of the right invariance of the vector field
	 $v(g)=dl_{g}(\Ad(g^{-1})(\psi_0))=dr_{g}(\psi_0),$
	since $ (\Ad(g^{-1})(\psi_0),(\Ad(g^{-1})(\psi_0))\equiv 1, $
	$$(\Ad(g)^{\ast}(\psi_0),v)=(\psi_0,\Ad(g)(v))=(\Ad(g^{-1})(\psi_0),v) \Rightarrow u(g)=\Ad(g^{-1})(\psi_0).$$
\end{proof}

\begin{theorem}
	\label{vf}
	If $v(g_0)\neq 0,$ $g_0\in G,$ then an integral curve of the vector field  $v(g), g\in G,$ with origin $g_0$ is a normal extremal 
	parametrized proportionally to arc length with the proportionality factor $|dl_{g_0^{-1}}(v(g_0))|.$
\end{theorem}

\begin{proof}
	Let $g(t),$ $t\in\mathbb{R},$ be an integral curve under consideration and set $\gamma=\gamma(t)=g_0^{-1}g(t),$ $t\in\mathbb{R}.$ 
	Then $\gamma$ is an integral curve of vector field $dl_{g_0^{-1}}v(g),$ $g\in G,$ with origin $e.$ Hence
	\begin{equation}
	\label{gam}
	\dot{\gamma}(t)=dl_{g_0^{-1}}\dot{g}(t)=dl_{g_0^{-1}}(dl_{g(t)}(u(g(t))))=dl_{\gamma(t)}(u(g(t))).
	\end{equation}
	In addition, 
	\begin{equation}
	\label{adp}
	\Ad(g(t))^{\ast}=\Ad(g_0\cdot \gamma(t))^{\ast}=\Ad(\gamma(t))^{\ast}\circ \Ad(g_0)^{\ast}.
	\end{equation}
	By definition, 
	$$u(g(t))=\Ad(g(t))^{\ast}(\psi_0)(w(g(t)))w(g(t)), $$
	$$\Ad(g(t))^{\ast}(\psi_0)(w(g(t)))=\max_{w\in U}\Ad(g(t))^{\ast}(\psi_0)(w),$$
	that by (\ref{adp}) can be rewrite as
	$$u(g(t))=\Ad(\gamma(t))^{\ast}(\psi_0')(w(g(t)),$$
	$$ \Ad(\gamma(t))^{\ast}(\psi_0')(w(g(t)))=\max_{w\in U}\Ad(\gamma(t))^{\ast}(\psi_0')(w),$$
	where $\psi_0'=\Ad(g_0)^{\ast}(\psi_0).$ 
	As a result of this and (\ref{gam}), we see  that $u(g(t))$ plays a role of $u(\gamma(t))$ for 
	constant covector $\psi_0'$ (instead of $\psi_0$). Due to point 2 of Theorem \ref{main} 
	the curve $\gamma(t)$ is a normal extremal parameterized proportionally to arc length with the proportionality factor $|dl_{g_0^{-1}}(v(g_0))|.$
	Then its left shift $g(t)=g_0\gamma(t)$  also has this property. 
\end{proof}

\begin{remark}
	Theorem \ref{vf} holds for left-invariant Riemannian metrics on (connected) Lie groups. In this case, $v(g_0)\neq 0$ for all 	$g_0\in G.$ 
\end{remark}

Let us choose a basis $\{e_1,\dots, e_n\}$  in $\mathfrak{g},$ assuming that $\{e_1,\dots, e_r\}$ is an 
orthonormal basis for the scalar product $\langle\cdot,\cdot\rangle$ on $\mathfrak{p}$  in case of left-invariant (sub)-Finsler metric. 
Define a scalar product $\langle\cdot,\cdot\rangle$ on $\mathfrak{g},$ considering  $\{e_1,\dots, e_n\}$ as its orthonormal basis.
Then each covector $\psi\in \mathfrak{g}^{\ast}$ can be considered as a vector in $\mathfrak{g},$ setting 
$\psi(v)=\langle \psi,v\rangle$ for every $v\in \mathfrak{g}.$ If $\psi=\sum_{i=1}^{n}\psi_ie_i,$ $v=\sum_{k=1}^nv_ke_k,$ then
$\psi(v)=\psi\cdot v,$ where $\psi$ and $v$ are corresponding vector-row and vector-column, $\cdot$ is the matrix multiplication.
If $l:\mathfrak{g}\rightarrow \mathfrak{g}$  is a linear map, then we denote by  $(l)$ its matrix in the basis
$\{e_1,\dots, e_n\}.$ 

\begin{proposition}
	$$(\Ad g)^{\ast}(\psi)=\psi(\Ad g),\quad g\in G,\,\,\psi\in \mathfrak{g}^{\ast},$$
	where  on the right hand side of the equality $\psi$  indicates the corresponding vector-row.
\end{proposition}

\begin{proof}
	Obviously, the identity 
	$$(\Ad g)^{\ast}(\psi)((\Ad g)^{-1}(v))=\psi(v)=\psi\cdot v$$
	holds. Therefore, it is enough to verify that for matrix $A:=(\Ad g)$ 
	$$(\psi A)(A^{-1}v)=\psi\cdot v.$$
	But it is obvious.
\end{proof}

If $g(t),$ $t\in \mathbb{R},$ is a normal geodesic of a left-invariant (sub-)Riemannian metric $d$ on a Lie group
$G,$ then $u(g(t))$ is the orthogonal projection onto $\mathfrak{p}$ of the vector $(\Ad g(t))^{\ast}(\psi_0)$ in the 
notation of Theorem \ref{main} for the scalar product $\langle\cdot,\cdot\rangle$ introduced above on  $\mathfrak{g}.$ 
This fact  and formula (\ref{hame}) imply

\begin{theorem}
	\label{kok}
	Every normal parameterized by arc length geodesic of left-invariant (sub-)Riemannian metric on a Lie group $G$ 
	issued from the unit is a solution of the following system of differential equations
	\begin{equation}
	\label{difur}
	\dot{g}(t)=dl_{g(t)}u(t),\,\,u(t)=\sum_{i=1}^r \psi_i(t)e_i,\,\,|u(0)|=1,\,\,
	\dot{\psi}_j(t)=\sum_{k=1}^{n}\sum_{i=1}^{r}c_{ij}^k\psi_i(t)\psi_k(t),
	\end{equation}
    where $ j=1,\dots, n,$ $c_{ij}^k$ are structure constants of Lie algebra $\mathfrak{g}$ in its basis $\{e_1,..., e_n\}.$ 
	In Riemannian case, $r=n$.
\end{theorem}

\begin{corollary}
	\begin{equation}
	\label{unit}
	|\dot{g}(t)| = |u(t)|\equiv 1,\quad t\in \mathbb{R}.
	\end{equation}
\end{corollary}

\begin{proof}
	The first equality in (\ref{unit}) is a consequence of the first equality in (\ref{difur}) and left invariance of the scalar product.
	Therefore, due to the equality $|u(0)|=1$, it suffices to prove that $\frac{d}{dt}\langle u(t),u(t)\rangle)=0.$ Now by (\ref{difur}),
	$$\frac{d}{dt}\langle u(t),u(t)\rangle=\left(\sum_{j=1}^r\psi_j(t)\psi_j(t)\right)'=2\sum_{j=1}^r\psi_j(t)\psi_j'(t)=\sum_{k=1}^n\sum_{i,j=1}^rc_{ij}^k\psi_i(t)\psi_j(t)\psi_k(t),$$
	which is zero by the skew symmetry of $c_{ij}^k$ with respect to subscripts.
\end{proof}

\begin{remark}
	In fact, the same equations for $\dot{\psi}_j(t)$ from (\ref{difur}) in a different interpretation were obtained in \cite{GK95} as 
	``normal equations''. Their derivation there uses more complicated concepts and techniques.
\end{remark}

\section{Lie groups with left-invariant Riemannian metrics of constant negative curvature}

The only Lie groups which do not admit left-invariant sub-Finsler metrics are commutative Lie groups and 
Lie groups $G_n,$ $n\geq 2$, consisting of parallel translations and homotheties (without rotations) of Euclidean 
space  $E^{n-1}$ \cite{BerGorb14}, \cite{Ber89}. Up to isomorphisms, Lie groups $G_n$ can be described as 
connected Lie groups every whose left-invariant Riemannian metric has constant negative sectional curvature \cite{Miln76}.

The group $G_n$, $n\geq 2$, is isomorphic to the group of real block matrices 
\begin{equation}
\label{elem}
g=(y,x):=\left(\begin{array}{cc} 
xE_{n-1} & y^{\prime}\\
0 & 1   c
\end{array}\right),
\end{equation}
where $E_{n-1}$ is unit matrix of order $n-1$, $y^{\prime}$ is a transposed 
$(n-1)-$vector--row $y$, 
$0$ is a zero $(n-1)-$vector--row, $x>0$.

It is clear that in vector notation the group operations have a form

\begin{equation}
\label{prod}
(y_1,x_1)\cdot (y_2,x_2)= x_1(y_2,x_2)+(y_1,0),\quad (y,x)^{-1}=x^{-1}(-y,1).
\end{equation}

Let $E_{ij}$, $i,j=1,\dots,n$, be a $(n\times n)$-matrix having 1 in the ith row and the jth column and 0 in all other. Matrices
\begin{equation}
\label{abc}
e_i=E_{in},\,\,i=1,\dots,n-1,\quad e_n=\sum\limits_{k=1}^{n-1}E_{kk}
\end{equation}
constitute a basis of Lie algebra $\frak{g}_n$ of the Lie group $G_n$. In addition,
$$[e_i,e_j]=0,\,\,i,j=1,\dots,n-1;\quad [e_n,e_i]=e_i,\,\,i=1,\dots,n-1$$
so all nonzero structure constants in the basis $\{e_1,\dots,e_n\}$  are equal
\begin{equation}
\label{strconst}
c_{ni}^i=- c_{in}^i=1,\,\quad i=1,\dots,n-1.
\end{equation}

Let  $(\cdot,\cdot)$ be a scalar product  on  $\frak{g}_n$ with the orthonormal basis  $e_1,\dots,e_n$.
Then we get left-invariant Riemannian metric $d$ on the Lie group $G_n$ of constant sectional curvature $-1$ \cite{Miln76}.

On the ground of Theorem \ref{kok} and (\ref{strconst}),  $\psi_i=\psi_i(t)$, $i=1,\dots,n$ are solutions of the  Cauchy problem
\begin{equation}
\label{sist0}
\left\{\begin{array}{c}
\dot{\psi}_i(t)=\psi_i(t)\psi_n(t),\,\,i=1,\dots,n-1,\quad\dot{\psi}_n(t)=-\sum\limits_{i=1}^{n-1}\psi_i^2(t); \\
\psi_i(0)=\varphi_i,\,\,i=1,\dots,n,\quad \sum\limits_{i=1}^n\varphi_i^2=1.
\end{array}\right.
\end{equation}
It follows from (\ref{sist0}) that
$$\ddot{\psi}_n(t)=-2\psi_n(t)\sum\limits_{i=1}^{n-1}\psi_i^2(t)=2\psi_n(t)\dot{\psi}_n(t)=
\left(\psi_n^2\right)^{\cdot}(t),$$
whence on the ground of initial data of the Cauchy problem (\ref{sist0}), it follows that
$$\dot{\psi}_n(t)=\psi_n^2(t)-1,\quad \psi_n(0)=\varphi_n.$$
Solving this Cauchy problem, we find that
$$\psi_n(t)=\frac{\varphi_n\ch t-\sh t}{\ch t-\varphi_n\sh t}.$$
Then on the base of (\ref{sist0}), for $i=1,\dots,n-1$,
$$\ln|\psi_i(t)|=\int\limits_0^t\frac{\varphi_n\ch\tau-\sh\tau}{\ch\tau-\varphi_n\sh\tau}d\tau+\ln{|\varphi_i|}=-\ln{|\ch t -\varphi_n\sh t|}+\ln{|\varphi_i|},\quad \text{if }\varphi_i\neq 0,$$
so
$$\psi_i(t)=\frac{\varphi_i}{\ch t-\varphi_n\sh t},\quad i=1,\dots,n-1,$$
and these formulae are true also when $\varphi_i=0$.

Consequently, on the ground of (\ref{difur}),
\begin{equation}
\label{u}
u(t)=\frac{1}{\ch t-\varphi_n\sh t}\left(\sum\limits_{i=1}^{n-1}\varphi_ie_i+\left(\varphi_n\ch t-\sh t\right)e_n\right).
\end{equation}

If $g\in G_n$ is defined by formula (\ref{elem}), $u=\sum\limits_{i=1}^nu_ie_i\in\frak{g}_n$, then
\begin{equation}
\label{gu}
gu=\left(\begin{array}{cc}
(xu_n)E_{n-1} & v\\
0 & 0
\end{array}\right),\quad v=(xu_1,\dots,xu_{n-1})^T.
\end{equation}

Therefore on the base of Theorem \ref{kok} and (\ref{u}) in the notation  (\ref{elem}), parametrized by arclength normal geodesic
 $g=g(t)$, $t\in\mathbb{R}$,  of the space $(G_n,d)$ with $g(0)=e$ is a solution of the  Cauchy problem
\begin{equation}
\label{coshi}
\left\{\begin{array}{c}
\dot{x}(t)=\frac{\varphi_n\ch t-\sh t}{\ch t-\varphi_n\sh t}x(t),\,\,\dot{y}_i(t)=\frac{\varphi_i}{\ch t-\varphi_n\sh t}x(t),\quad i=1,\dots,n-1, \\
x(0)=1,\quad y_i(0)=0,\,\,i=1,\dots,n-1.
\end{array}\right.
\end{equation}
Solving the problem, we find
\begin{equation}
\label{xy}
x(t)=\frac{1}{\ch t-\varphi_n\sh t},\quad y_i(t)=\int\limits_0^t\frac{\varphi_idt}{(\ch t-\varphi_n\sh t)^2}=\frac{\varphi_i\sh t}{\ch t-\varphi_n\sh t}.
\end{equation}

This implies that
\begin{equation}
\label{expon}
x(t)=e^{\pm t},\quad y_i(t)\equiv 0,\quad i=1,\dots n-1,\quad\mbox{if}\quad \varphi_n = \pm 1.
\end{equation}

Let $\varphi_n^2<1$. Let us show that for any $t\in\mathbb{R}$, the equality 
\begin{equation}
\label{u0}
\sum\limits_{i=1}^{n-1}(y_i(t)-a_i)^2+x^2(t)=\sum\limits_{i=1}^{n-1}a_i^2+1
\end{equation}
holds, where $a_i$, $i=1,\dots,n-1$, are real numbers such that
\begin{equation}
\label{ai}
\sum\limits_{i=1}^{n-1}a_i\varphi_i=\varphi_n.
\end{equation}

We introduce a function $f(t)=\sum\limits_{i=1}^{n-1}(y_i(t)-a_i)^2+x^2(t)$. Due to initial data (\ref{coshi}),  
$f(0)=\sum\limits_{i=1}^{n-1}a_i^2+1$. On the ground of (\ref{coshi}), (\ref{xy}) and last equation in (\ref{sist0}), we get
$$\frac{1}{2}f^{\prime}(t)=\sum\limits_{i=1}^{n-1}(y_i(t)-a_i)\dot{y}_i(t)+x(t)\dot{x}(t)=
\sum\limits_{i=1}^{n-1}\left(\frac{\varphi_i\sh t}{\ch t-\varphi_n\sh t}-a_i\right)\varphi_i+\frac{\varphi_n\ch t-\sh t}{\ch t-\varphi_n\sh t}=$$
$$\frac{\sh t\left(\sum\limits_{i=1}^{n-1}\varphi_i^2-1\right)+\varphi_n\ch t}{\ch t-\varphi_n\sh t}-\sum\limits_{i=1}^{n-1}a_i\varphi_i=\varphi_n-\sum\limits_{i=1}^{n-1}a_i\varphi_i=0.$$
Consequently, $f(t)\equiv f(0)$ and the equality (\ref{u0}) is proved.

It is easy to check that the equality (\ref{ai}) holds for
\begin{equation}
\label{pai}
a_i=\varphi_i\varphi_n/(1-\varphi_n^2),\quad i=1,\dots, n-1;\quad\mbox{moreover}\quad \sum\limits_{i=1}^{n-1}a_i^2+1 = \frac{1}{1-\varphi_n^2}.
\end{equation}
These numbers $a_i$ are obtained as halves of sums of limits $y_i(t)$ when
 $t\rightarrow +\infty$ and $t\rightarrow -\infty$, which are equal to $\varphi_i/(1-\varphi_n)$ and  
  $-\varphi_i/(1+\varphi_n)$ respectively.

Formulae (\ref{prod}) show that the group  $G_n$ is a simply transitive  isometry group of the famous 
Poincare's model  of the Lobachevskii space $L^n$ in the half space  $\mathbb{R}^n_+$ with metric 
$ds^2=(\sum_{k=1}^{n-1}dy_k^2+dx^2)/x^2$.

The above results, including formulae  (\ref{xy}), (\ref{expon}), (\ref{pai}), show that geodesics 
of the space $L^n$ in this model, passing through the point $(0,\dots,0,1),$ are semi-straights or 
semi-circles (with centers $(a_1,\dots,a_{n-1},0)$ and radii $1/{\sqrt{1-\varphi_n^2}}$, (\ref{pai})), orthogonal to the hyperplane
$\mathbb{R}^{n-1}\times \{0\}.$  
Since all other geodesics are obtained by left shifts on the group, in other words, by indicated parallel
translations and homotheties of this model, then also all straights and semi-circles, orthogonal to the hyperplane
 $\mathbb{R}^{n-1}\times \{0\},$ are geodesics of the space $L^n.$

We got a well-known description of geodesics in this  Poincare's model.

\vspace{0.2cm}

Now let us look what the vector field method gives us for the problem.

Every vector $\psi\in\frak{g}_n$ can be considered as a covector $\frak{g}^{\ast}$, setting $\psi(v)=(\psi,v)$ for $v\in\frak{g}_n$.
Then any (co)vector $\psi_0$ from Theorem \ref{main} has a form
$$\psi_0=\sum\limits_{i=1}^n\varphi_ie_i,\quad\sum\limits_{i=1}^n\varphi_i^2=1.$$

Let $w=\sum\limits_{i=1}^nw_ie_i\in\frak{g}_n$, $g\in G_n$ is defined by formula (\ref{elem}). 
It is easy to see that
$$\Ad(g)(w)=gwg^{-1}=\sum\limits_{i=1}^{n-1}(w_ix-w_ny_i)e_i+w_ne_n,$$
$$(\psi_0,\Ad(g)(w))=\sum\limits_{i=1}^{n-1}(w_ix-w_ny_i)\varphi_i+w_n\varphi_n=
x\sum\limits_{i=1}^{n-1}\varphi_iw_i+\left(\varphi_n-\sum\limits_{i=1}^{n-1}\varphi_iy_i\right)w_n.$$
It is clear that
$$u(g)=x\sum\limits_{i=1}^{n-1}\varphi_ie_i+\left(\varphi_n-\sum\limits_{i=1}^{n-1}\varphi_iy_i\right)e_n,$$
$$v(g)=gu(g)=x\sum\limits_{i=1}^{n}u_ie_i=x^2\sum\limits_{i=1}^{n-1}\varphi_ie_i+x\left(\varphi_n-\sum\limits_{i=1}^{n-1}\varphi_iy_i\right)e_n.$$

Thus geodesic $g=g(t)$, $t\in\mathbb{R}$, with $g(0)=e$ is a solution of the Cauchy problem
\begin{equation}
\label{coshi2}
\left\{\begin{array}{c}
\dot{x}(t)=\left(\varphi_n-\sum\limits_{i=1}^{n-1}\varphi_iy_i(t)\right)x(t),\quad\dot{y}_i(t)=\varphi_ix^2(t),\,\,i=1,\dots,n-1, \\
x(0)=1,\quad y_i(0)=0,\,\,i=1,\dots,n-1.
\end{array}\right.
\end{equation}
Dividing the first equation in (\ref{coshi2}) by $x(t),$ we get on the left hand side the derivative of the function $\ln x(t):= z(t).$
Differentiating both sides of the resulting equation and using the second equation in (\ref{coshi2}) and the equality 
$\sum\limits_{i=1}^n\varphi_i^2=1$, we get
$$\ddot{z}(t)=-\sum\limits_{i=1}^{n-1}\varphi_i^2x^2(t)= -(1-\varphi_n^2)e^{2z(t)},\quad z(0)=0,\,\,\dot{z}(0)=\varphi_n.$$

If $\varphi_n=\pm 1$ then $\ddot{z}(t)\equiv 0$ and due to the initial data and the second equation in (\ref{coshi2}), we get $z(t)=\pm t,$ 
$x(t)=e^{\pm t},$ $y_i(t)\equiv 0,$ $i=1,\dots,n-1$.

Let $0\leq \varphi_n^2 < 1.$ Let us multiply both sides of the resulting equation by $2\dot{z}.$ Then
$$2\dot{z}\ddot{z}=-(1-\varphi_n^2)e^{2z}2\dot{z},\quad d(\dot{z})^2=-(1-\varphi_n^2)e^{2z}d{(2z)},\quad \dot{z}^2=-(1-\varphi_n^2)e^{2z}+ C.$$ 
Taking into account  the initial conditions for $z(t),$ we get $C=1$ and $\dot{z}(t)^2=1-(1-\varphi_n^2)e^{2z(t)}.$ The expression on the 
right is positive for $t$ sufficiently close to zero. Therefore, with these  $t,$ we get $$\dot{z}(t)=\pm\sqrt{ 1-(1-\varphi_n^2)e^{2z(t)}},$$
where the sign coincides with the sign of $\varphi_n,$ if $\varphi_n\neq 0.$ Separating variables, we get
$$dt= \frac{\pm dz}{\sqrt{ 1-(1-\varphi_n^2)e^{2z}}}= \frac{\pm dz}{e^z\sqrt{1-\varphi_n^2}\sqrt{(e^{-2z}/(1-\varphi_n^2))-1}}=$$
$$ \frac{\mp d(e^{-z}/\sqrt{1-\varphi_n^2})}{\sqrt{(e^{-2z}/(1-\varphi_n^2))-1}}=
\mp d\left(\arch\left(\frac{e^{-z}}{\sqrt{1-\varphi_n^2}}\right)\right),$$
$$\pm  \arch\left(\frac{e^{-z}}{\sqrt{1-\varphi_n^2}}\right)=c-t, \quad c =  \arch\left(\frac{1}{\sqrt{1-\varphi_n^2}}\right).$$
The applying $\ch$ to the left and right sides of the resulting equality gives
$$\frac{e^{-z(t)}}{\sqrt{1-\varphi_n^2}}=\ch c\ch t- \sh c\sh t= \frac{\ch t -\varphi_n \sh t}{\sqrt{1-\varphi_n^2}}.$$
Consequently, when $ t $ are sufficiently close to zero,
$$x(t)=e^{z(t)}=\frac{1}{\ch t- \varphi_n\sh t}.$$
Since the right sides of the system of differential equations (\ref{coshi2}) are real analytic, 
this equality is true for all $t\in \mathbb{R}.$ 
We obtain from this and the second system in  (\ref{coshi2})  the same solutions $y_i(t),$ $t\in\mathbb{R},$ $i=1,\dots, n-1,$ as in (\ref{xy}).

Using formulae (\ref{prod}) and (\ref{xy}) for $x=x(t)$, $y_i=y_i(t),$ we shall find a formula for distances between group elements, 
or, which is the same, between points of the Lobachevsky space in Poincare's model under consideration.
We obtain from (\ref{xy}) 
$$\frac{1}{x}=\ch t-\varphi_n\sh t,\quad x=\frac{\ch t +\varphi_n\sh t}{\ch^2t-\varphi_n^2\sh^2t}=\frac{\ch t +\varphi_n\sh t}{1 + (1-\varphi_n^2)\sh^2t},$$
$$\sum_{i=1}^{n-1}(y_i/x)^2=\sh^2t\sum_{i=1}^{n-1}\varphi_i^2=(1-\varphi_n^2)\sh^2t,$$
$$\ch t+\varphi_n\sh t=\frac{x}{x^2}\left(x^2+\sum_{i=1}^{n-1}y_i^2\right)=\frac{1}{x}\left(x^2+\sum_{i=1}^{n-1}y_i^2\right),$$
$$\ch t=\frac{1}{2x}\left(1+x^2+\sum_{i=1}^{n-1}y_i^2\right),\quad d((0,1),(y,x))=\arch\left[\frac{1}{2x}\left(1+x^2+\sum_{i=1}^{n-1}y_i^2\right)\right].$$
Now by  (\ref{prod}), the last formula, and left-invariance of metric $d$,
$$(y_1,x_1)^{-1}(y_2,x_2)=x_1^{-1}(-y_1,1)(y_2,x_2)=(x_1^{-1}(y_2-y_1), x_1^{-1}x_2),$$
$$d((y_1,x_1),(y_2,x_2))=d((0,1), (x_1^{-1}(y_2-y_1), x_1^{-1}x_2))=$$
$$\arch\left[\frac{x_1}{2x_2}\left(1+ \frac{x_2^2}{x_1^2}+\frac{1}{x_1^2}\sum_{i=1}^{n-1}(y_{2,i}-y_{1,i})^2\right)\right]=$$
\begin{equation}
\label{dist}
\arch\left[\frac{1}{2x_1x_2}\left(x_1^2+ x_2^2+ \sum_{i=1}^{n-1}(y_{2,i}-y_{1,i})^2\right)\right]= d((y_1,x_1),(y_2,x_2)).
\end{equation}

\section{ The three--dimensional Heisenberg group}

This Heisenberg group is a nilpotent Lie group of upper--triangular matrices
\begin{equation}
\label{heis}
H=\left\{h=\left(\begin{array}{ccc}
1& x & z\\
0& 1 & y \\
0 & 0 & 1
\end{array}\right)\right\},\,\, x,y,z \in \mathbb{R}.
\end{equation}
It is easy to compute that
\begin{equation}
\label{hm}
h^{-1}=\left(\begin{array}{ccc}
1& -x & xy- z\\
0& 1 & -y \\
0 & 0 & 1
\end{array}\right).
\end{equation}
Clearly, $H$ is naturally diffeomorphic to $\mathbb{R}^3$ and $H$ is a connected Lie group with respect to this differential structure.
Matrices 
\begin{equation}
\label{base}
e_1=\left(\begin{array}{ccc}
0& 1 & 0\\
0& 0 & 0 \\
0 & 0 & 0
\end{array}\right),  \quad
e_2=\left(\begin{array}{ccc}
0& 0 & 0\\
0& 0 & 1 \\
0 & 0 & 0
\end{array}\right),\quad
e_3=\left(\begin{array}{ccc}
0& 0 & 1\\
0& 0 & 0 \\
0 & 0 & 0
\end{array}\right)
\end{equation}
constitute a basis of Lie algebra $\mathfrak{h}$ of Heisenberg group  $H$. In addition, 
$$[e_1,e_2]= e_1e_2- e_2e_1= e_3.$$
Hence the vector subspace $\mathfrak{p}\subset \mathfrak{h}$ with basis  $\{e_1,e_2\}$ generates
$\mathfrak{h}.$ 

Thus the triple $(H,\mathfrak{h},\mathfrak{p})$ satisfies all conditions of Theorems \ref{contr} and \ref{topt}.

Let us search for all geodesics of the problem from Theorem \ref{topt}.
They are all normal by Theorem \ref{norm}, and we can use Theorem \ref{main}.

Let us define a scalar product $(\cdot,\cdot)$ on $\mathfrak{h}$ with orthonormal basis $\{e_1,e_2,e_3\}$. 
Then each vector $\psi\in \mathfrak{h}$ can be considered as a covector from $\mathfrak{h}^{\ast},$ if we set
$\psi(v)=(\psi,v)$ for $v\in\mathfrak{h}.$ Then any (co)vector $\psi_0$ from Theorem \ref{main} has a form
\begin{equation}
\label{psn}
\psi_0= \cos\xi e_1+ \sin\xi e_2+ \beta e_3,\quad \xi, \beta\in \mathbb{R}.
\end{equation}
Let 
$$v=\sum_{k=1}^2v_ke_k=\left(\begin{array}{ccc}
0& v_1 & 0\\
0& 0 & v_2 \\
0 & 0 & 0
\end{array}\right),\quad v\in \mathfrak{p},\,\, v_k\in\mathbb{R},\,\, k=1,2.$$
Using formulae (\ref{heis}), (\ref{hm}), we get
$$Ad(h)(v)=hvh^{-1}= \left(\begin{array}{ccc}
0& v_1 & -yv_1+xv_2\\
0& 0 & v_2 \\
0 & 0 & 0
\end{array}\right),
$$
$$(\psi_0,\Ad(h)(v))=\cos\xi v_1+\sin\xi v_2+ \beta( -yv_1+xv_2)=$$
$$(\cos\xi -\beta y)v_1+ (\sin\xi +\beta x)v_2.$$
It is clear that
$$u(h)=(\cos\xi -\beta y)e_1+ (\sin\xi +\beta x)e_2$$
and so a geodesic is an integral curve of the vector field 
$$v(h)= hu(h)=(\cos\xi -\beta y)e_1+  (\sin\xi +\beta x)e_2+ x (\sin\xi +\beta x)e_3.$$
Therefore $h(t)$ is a solution of the Cauchy problem 
\begin{equation}
\label{sist}
\left\{\begin{array}{l}
\dot{x}=\cos\xi -\beta y, \\
\dot{y}=\sin\xi +\beta x, \\
\dot{z}= x (\sin\xi +\beta x)(=x\dot{y}) \\
\end{array}\right.
\end{equation}
with initial data $x(0)=y(0)=z(0)=0$.

Let us turn to {\it the coordinate system $\tilde{x},\tilde{y},\tilde{z}$  of the first kind} on the Lie group $H:$
$$\exp\left(\begin{array}{ccc}
0& x & z\\
0& 0 & y \\
0 & 0 & 0
\end{array}\right)=\left(\begin{array}{ccc}
1& x & z+(xy)/2\\
0& 1 & y \\
0 & 0 & 1
\end{array}\right).
$$
Hence $\tilde{x}=x, \tilde{y}=y, \tilde{z}=z-(xy)/2.$

It is easy to see that for $\beta=0$ we get
$$x(t)=(\cos\xi)t,\,\, y(t)=(\sin\xi)t,\,\, z(t)=\frac{1}{2}\cos\xi\sin\xi t^2,\,\,\tilde{z}(t)\equiv 0,\,\, t\in\mathbb{R},$$
and geodesic is a $1$--parameter subgroup
$$g(t)=\exp(t(\cos\xi e_1 + \sin\xi e_2)),\,\, t\in\mathbb{R}.$$

If $\beta\neq 0$, the calculations are more difficult:
$$\ddot{x}=-\beta\dot{y}=-\beta(\sin\xi + \beta x)=-\beta^2x-\beta\sin\xi,$$
$$x(t)=C_1\cos\beta t + C_2\sin\beta t -\frac{\sin\xi}{\beta}.$$
Since $x(0)=0,$ $\dot{x}(0)=\cos\xi$, then $C_1=(\sin\xi)/\beta,$ $C_2=(\cos\xi)/\beta,$ 
\begin{equation}
\label{xt}
x(t)=\frac{1}{\beta}(\sin\xi\cos \beta t + \cos\xi \sin\beta t-\sin\xi)=\frac{1}{\beta}(\sin(\xi+ \beta t)-\sin\xi);
\end{equation}
$$\ddot{y}=\beta\dot{x}=\beta(\cos\xi - \beta y)=-\beta^2y+\beta\cos\xi,$$
$$y(t)=C_1\cos\beta t + C_2\sin\beta t +\frac{\cos\xi}{\beta}.$$
Since $y(0)=0,$ $\dot{y}(0)=\sin\xi$, then $C_1=-(\cos\xi)/\beta,$ $C_2=(\sin\xi)/\beta,$ 
\begin{equation}
\label{yt}
y(t)=\frac{1}{\beta}(-\cos\xi\cos \beta t + \sin\xi\sin\beta t+\cos\xi)=\frac{1}{\beta}(-\cos(\xi+ \beta t)+\cos\xi),
\end{equation}
$$ \tilde{z}'= \dot{z}-\frac{(xy)'}{2}=x\dot{y}-\frac{1}{2}(\dot{x}y+x\dot{y})=\frac{1}{2}(x\dot{y}-\dot{x}y)=$$
$$\frac{1}{2\beta}[(\sin(\xi+ \beta t)-\sin\xi)\sin(\xi+\beta t)-\cos(\xi+\beta t)(-\cos(\xi+ \beta t)+\cos\xi)]=$$
$$\frac{1}{2\beta}[1-(\sin\xi\cdot\sin(\xi+\beta t)+\cos(\xi+\beta t)\cos\xi)]=\frac{1}{2\beta}(1-\cos\beta t)= \tilde{z}'.$$
Since $\tilde{z}(0)=0$ then
\begin{equation}
\label{zt}
\tilde{z}(t)= \frac{1}{2\beta}\left(t-\frac{\sin\beta t}{\beta}\right), t\in\mathbb{R}.
\end{equation}

It follows from equalities  (\ref{xt}), (\ref{yt}), ($\ref{zt}$) that the projection of geodesic  $g=g(t)$ onto the plane $x,y$ is 
{\it a circle with radius $1/|\beta|$ and center  $(1/\beta)(-\sin\xi,\cos\xi)$,  $T=2\pi/|\beta|$ is a circulation period}, 
while $\tilde{z}(t),$ $t\in\mathbb{R},$ does not depend on the parameter  $\xi.$
Therefore, if we fix  $\beta\neq 0$ then for different  $\xi$ all geodesic segments 
$g(\beta,\xi,t), 0\leq t\leq 2\pi/|\beta|,$ start at $e$ and finish at the same point.
It follows from the existence of the shortest arcs, Theorem \ref{topt}, PMP and our calculations that if $\beta=0$
(respectively, $\beta\neq 0$) then every segment (respectively,  the length of such segment is less or equal to 
$T=2\pi/|\beta|$) of these geodesics is a shortest arc. There is no other geodesic or shortest arc except 
indicated above and their left shifts.

\section{Controls for left-invariant sub-Riemannian  metrics on $SO(3)$}
\label{so3}

It is well known that every two--dimensional vector subspace $\mathfrak{p}$ of Lie algebra $(\mathfrak{so}(3),[\cdot,\cdot])$ 
of the Lie group $SO(3)$ generates $\mathfrak{so}(3).$ Moreover, there exists a basis $\{e_1,e_2\}$ of the space  
$\mathfrak{p}$ such that  $[e_2,e_3]=e_1,$ $[e_3,e_1]=e_2$ for the vector 
$e_3=[e_1,e_2]$. Let $(\cdot,\cdot)$ be a scalar product on $\mathfrak{so}(3)$ with orthonormal basis
 $\{e_1,e_2,e_3\}.$ Then if a scalar product $\langle\cdot,\cdot\rangle$
on $\mathfrak{p}$ defines a left-invariant sub-Riemannian metric $d$ on the Lie group $G=SO(3),$ then there exists
a basis $\{v,w\}$ in  $\mathfrak{p}$  that is orthonormal relative to $\langle\cdot,\cdot\rangle,$ orthogonal relative to 
$(\cdot,\cdot),$ and such that $(v,v)=a^2\leq b^2=(w,w),$ $[v,w]=(ab)e_3,$ where $0< a\leq b.$ Let $v,w$ be new vectors $e_1,e_2.$ Then
\begin{equation}
\label{e1e2e3}
[e_1,e_2]=(ab)e_3,\,\,[e_3,e_1]=(b/a)e_2,\,\,[e_2,e_3]=(a/b)e_1,\,\, 0< a\leq b.
\end{equation}

It follows from (\ref{e1e2e3}) that all nonzero structure constants are 
$$c_{12}^3=- c_{21}^3= ab,\,\, c_{31}^2=- c_{13}^2=b/a,\,\,c_{23}^1=-c_{32}^1= a/b.$$

Let $g(t)$, $t\in\mathbb{R}$, be a  geodesic of the space $(SO(3),d)$,  parametrized by arclength, and $g(0)=e$. 
On the ground of Theorem \ref{kok},
$$g^{\prime}(t)=g(t)u(t),\quad u(t)=\psi_1(t)e_1+\psi_2(t)e_2,$$
where
\begin{equation}
\label{psi}
\psi^{\prime}_1(t)=-ab\psi_2(t)\psi_3(t),\quad \psi^{\prime}_2(t)=ab\psi_1(t)\psi_3(t),\quad
\psi^{\prime}_3(t)=\frac{a^2-b^2}{ab}\psi_1(t)\psi_2(t).
\end{equation}
Since $|u(t)|\equiv 1$ then $\psi_1(t)=\cos\xi(t)$, $\psi_2(t)=\sin\xi(t)$ and (\ref{psi}) is written as
$$-\sin\xi(t)\dot{\xi}(t)=-ab\sin\xi(t)\psi_3(t), \quad \cos\xi(t)\dot{\xi}(t)=ab\cos\xi(t)\psi_3(t),$$
$$ \psi'_3(t)=\frac{a^2-b^2}{ab}\cos\xi(t)\sin\xi(t).$$
Then $\psi_3(t)=\frac{1}{ab}\xi^{\prime}(t)$ and $\xi=\xi(t)$ is a solution of the differential equation
\begin{equation}
\label{xid}
\xi^{\prime\prime}(t)=\frac{a^2-b^2}{2}\sin 2\xi(t).
\end{equation}

If $a=b$ then $\xi''(t)=0,$ $\xi'(t)=\const=\beta.$ Then geodesics are obtained from geodesics in the case of $a=b=1$ with 
the change the parameter $s$  by the parameter $t=s/a.$ Geodesics, shortest arcs, the distance $d,$ the cut locus and conjugate 
sets for geodesics in the case of $a=b=1$ are found in papers \cite{BZ15} and  \cite{BZ151}.

The case  $0<a<b$  is reduced to the case $a^2-b^2=-1$ by proportional change of the metric $d$.
Then the variable $\omega(t):=2\xi(t)$ allows us to rewrite the equation as the mathematical pendulum equation
\begin{equation}
\label{omega}
\omega^{\prime\prime}(t)=-\sin\omega(t).
\end{equation} 

In \cite{BS16},  I.Yu.~Beschastnyi and Yu.L.~Sachkov  studied geodesics of left-invariant sub-Riemannian metrics on the Lie group 
$SO(3)$ and gave estimates for the cut time and the metric diameter. Under replacement  $b^2-a^2$ by  $a^2$ and $\xi$ by $\psi,$ 
the equation (\ref{xid}) coincides with the equation (2.4) from their paper, obtained by another method.

\section{To search for geodesics of a sub-Riemannian  metric on $SH(2)$}

The Lie group $SH(2)$ consists of all matrices of a form 
\begin{equation}
\label{matr3}
g=\left(\begin{array}{cc}
A& v\\
0& 1\end{array}\right);\quad A=\left(\begin{array}{cc}
\ch\varphi & \sh\varphi\\
\sh\varphi & \ch\varphi 
\end{array}\right),\quad v=\left(\begin{array}{c}
x\\
y
\end{array}\right) \in \mathbb{R}^2. 
\end{equation}
It is not difficult to see that
\begin{equation}
\label{inv}
g^{-1}=
\left(\begin{array}{cc}
A & v\\
0& 1
\end{array}\right)^{-1}=\left(\begin{array}{cc}
A^{-1}& -A^{-1}v\\
0& 1
\end{array}\right).
\end{equation}

Clearly, matrices
\begin{equation}
\label{abc}
e_1=\left(\begin{array}{ccc}
0& 1 & 0\\
1& 0 & 0 \\
0 & 0 & 0
\end{array}\right),\quad
e_2=\left(\begin{array}{ccc}
0& 0 & 1\\
0& 0 & 0 \\
0 & 0 & 0
\end{array}\right),\quad
e_3=\left(\begin{array}{ccc}
0& 0 & 0\\
0& 0 & 1 \\
0 & 0 & 0
\end{array}\right)
\end{equation}
constitute a basis of Lie algebra $\mathfrak{sh}(2).$ In addition, 
\begin{equation}
\label{abca}
[e_1,e_2]=e_3,\quad [e_2,e_3]=0,\quad [e_1,e_3]=e_2.
\end{equation}

Let us define a scalar product $\langle\cdot,\cdot\rangle$ on $\mathfrak{sh}(2)$  with orthonormal basis
$\{e_1,\,e_2,\,e_3\}$ and the subspace $\mathfrak{p}$ with orthonormal basis $\{e_1,\,e_2\}$ 
generating Lie algebra $\mathfrak{sh}(2)$.
Thus a left-invariant sub-Riemannian metric  $d$ is defined on the Lie group  $SH(2).$

Let us take a (co)vector $\psi_0= \cos\alpha e_1+ \sin\alpha e_2+\beta e_3\in\frak{sh}(2)$. We calculate
$$\psi_g(w)=\langle\psi_g,w\rangle=\langle\psi_0,gwg^{-1}\rangle\quad g\in SH(2),\,\,w= w_1e_1+w_2e_2\in\mathfrak{p}.$$
$$gwg^{-1}=\tiny{\left(\begin{array}{ccc}
	\ch\varphi & \sh\varphi & x\\
	\sh\varphi & \ch\varphi & y \\
	0 & 0 & 1
	\end{array}\right)
	\left(\begin{array}{ccc}
	0 & w_1 & w_2 \\
	w_1 & 0 & 0 \\
	0 & 0 & 0
	\end{array}\right)
	\left(\begin{array}{ccc}
	\ch\varphi & -\sh\varphi & -x\ch\varphi+y\sh\varphi\\
	-\sh\varphi & \ch\varphi & x\sh\varphi-y\ch\varphi \\
	0 & 0 & 1
	\end{array}\right)}$$
$$=w_1e_1+(-w_1y+w_2\ch\varphi)e_2+ (-w_1x+w_2\sh\varphi)e_3,$$
$$\psi_g(v)=w_1\cos\alpha+(-w_1y+w_2\ch\varphi)\sin\alpha+(-w_1x+ w_2\sh\varphi)\beta=$$
$$w_1(\cos \alpha - e\sin\alpha-\beta x)+ w_2(\ch\varphi\sin\alpha + \beta\sh\varphi).$$
Therefore,
$$u(g)=(\cos\alpha-y\sin\alpha-\beta x)e_1+(\sin\alpha\ch\varphi+\beta\sh\varphi)e_2,$$
$$v(g)=gu(g)=\tiny{\left(\begin{array}{ccc}
	\ch\varphi & \sh\varphi & x\\
	\sh\varphi & \ch\varphi & y \\
	0 & 0 & 1
	\end{array}\right)\left(\begin{array}{ccc}
	0 & \cos\alpha-y\sin\alpha-\beta x & \sin\alpha\ch\varphi+\beta\sh\varphi\\
	\cos\alpha-y\sin\alpha-\beta x & 0 & 0 \\
	0 & 0 & 0
	\end{array}\right)}$$
$$=\tiny{\left(\begin{array}{ccc}
	\sh\varphi(\cos\alpha-y\sin\alpha-\beta x) & \ch\varphi(\cos\alpha-y\sin\alpha-\beta x) & \ch\varphi(\sin\alpha\ch\varphi+\beta\sh\varphi) \\
	\ch\varphi(\cos\alpha-y\sin\alpha-\beta x) & \sh\varphi(\cos\alpha-y\sin\alpha-\beta x) & 
	\sh\varphi(\sin\alpha\ch\varphi+\beta\sh\varphi) \\
	0 & 0 & 0
	\end{array}\right)}.$$
Hence integral curves of vector field $v(g),$ $g\in SH(2),$ satisfy the system of differential equations
\begin{equation}
\label{sistem}
\left\{\begin{array}{l}
\dot{\varphi}=\cos\alpha-y\sin\alpha-\beta x, \\
\dot{x}=\ch\varphi(\sin\alpha\ch\varphi+\beta\sh\varphi), \\
\dot{y}=\sh\varphi(\sin\alpha\ch\varphi+\beta\sh\varphi). \\
\end{array}\right.
\end{equation}
The geodesic $g(t),$ $t\in\mathbb{R}$, with $g(0)=e$ is a solution of this system with initial data $\varphi(0)=x(0)=y(0)=0$.
In this case, $|u(g(t))|\equiv 1$, i.e.
\begin{equation}
\label{m1}
g(t)\in M_1=\{(\sin\alpha\ch\varphi+\beta\sh\varphi)^2+(\cos\alpha-y\sin\alpha-\beta x)^2=1\}\subset SH(2). 
\end{equation}
Therefore there exists a differentiable function $\gamma=\gamma(t)$ such that
\begin{equation}
\label{s0}
\cos\frac{\gamma}{2}=\sin\alpha\ch\varphi+\beta\sh\varphi,\quad\sin\frac{\gamma}{2}=\cos\alpha-y\sin\alpha-\beta x.
\end{equation}
Since $\varphi(0)=x(0)=y(0)=0,$ then we can assume that $\gamma(0)=\pi-2\alpha$.

On the ground of (\ref{s0}) the sistem (\ref{sistem}) is written in the form
\begin{equation}
\label{sistem1}
\left\{\begin{array}{l}
\dot{\varphi}=\sin{\frac{\gamma}{2}}, \\
\dot{x}=\cos{\frac{\gamma}{2}}\ch\varphi, \\
\dot{y}= \cos{\frac{\gamma}{2}}\sh\varphi.
\end{array}\right.
\end{equation}

Differentiating the first and the second equalities in (\ref{s0}) and using (\ref{sistem1}), we get
$$-\frac{\dot{\gamma}}{2}\sin\frac{\gamma}{2}= (\sin\alpha\sh\varphi+\beta\ch\varphi)\dot{\varphi}=
\sin{\frac{\gamma}{2}}\left(\sin\alpha\sh\varphi+\beta\ch\varphi\right),$$
$$\frac{\dot{\gamma}}{2}\cos\frac{\gamma}{2}=-\dot{y}\sin\alpha-\beta\dot{x}=
-\cos{\frac{\gamma}{2}}\left(\sin\alpha\sh\varphi+\beta\ch\varphi\right),$$
whence
$$\dot{\gamma}=-2(\sin\alpha\sh\varphi+ \beta\ch\varphi),\quad \dot{\gamma}(0)=-2\beta.$$
Consequently, on the ground of the first equality in (\ref{s0}) and (\ref{sistem1})
$$\ddot{\gamma}=-2(\sin\alpha\ch\varphi+\beta\sh\varphi)\dot{\varphi}=-2\cos\frac{\gamma}{2}\sin\frac{\gamma}{2}=-\sin\gamma.$$
We got the  mathematical pendulum equation. In paper \cite{BSB14} this equation together with equations (\ref{sistem1}) are obtained by 
another method replacing $\varphi$ with $z.$  

\section{To search for geodesics of a sub-Riemannian  metric on $SE(2)$}

\label{se2}

The Lie group $SE(2)$ is isomorphic to the group of matrices of a form
\begin{equation}
\label{matr4}
\left(\begin{array}{cc}
A& v\\
0& 1\end{array}\right); \quad A=\left(\begin{array}{cc}
\cos\varphi & -\sin\varphi\\
\sin\varphi & \cos\varphi 
\end{array}\right),\quad v=\left(\begin{array}{c}
x\\
y
\end{array}\right) \in \mathbb{R}^2. 
\end{equation}
The same formula (\ref{inv}) is true.

It is clear that matrices
\begin{equation}
\label{abc}
e_1=\left(\begin{array}{ccc}
0& -1 & 0\\
1& 0 & 0 \\
0 & 0 & 0
\end{array}\right),\quad
e_2=\left(\begin{array}{ccc}
0& 0 & 1\\
0& 0 & 0 \\
0 & 0 & 0
\end{array}\right),\quad
e_3=\left(\begin{array}{ccc}
0& 0 & 0\\
0& 0 & 1 \\
0 & 0 & 0
\end{array}\right)
\end{equation}
constitute a basis of Lie algebra $\mathfrak{se}(2).$ In addition,
\begin{equation}
\label{abca}
[e_1,e_2]=e_3,\quad [e_1,e_3]=-e_2,\quad [e_2,e_3]= 0.
\end{equation}

Let us define a scalar product $\langle\cdot,\cdot\rangle$ on $\mathfrak{se}(2)$  with orthonormal basis
$\{e_1,\,e_2,\,e_3\}$ and the subspace $\mathfrak{p}$ with orthonormal basis $\{e_1,\,e_2\}$ generating 
Lie algebra $\mathfrak{se}(2)$. Thus a left-invariant sub-Riemannian metric  $d$ is defined on the Lie group  $SE(2)$
(see \cite{Ber941}--\cite{S10} and other papers).  

Let us take a (co)vector $\psi_0= \cos\alpha e_1+ \sin\alpha e_2+\beta e_3\in\frak{se}(2)$.
We calculate 
$$\psi_g(w)=\langle\psi_g,w\rangle=\langle\psi_0,gwg^{-1}\rangle\quad g\in SH(2),\,\,w= w_1e_1+w_2e_2\in\mathfrak{p}.$$
$$gwg^{-1}=\tiny{\left(\begin{array}{ccc}
	\cos\varphi & -\sin\varphi & x\\
	\sin\varphi & \cos\varphi & y \\
	0 & 0 & 1
	\end{array}\right)
	\left(\begin{array}{ccc}
	0 & -w_1 & w_2 \\
	w_1 & 0 & 0 \\
	0 & 0 & 0
	\end{array}\right)
	\left(\begin{array}{ccc}
	\cos\varphi & \sin\varphi & -x\cos\varphi-y\sin\varphi\\
	-\sin\varphi & \cos\varphi & x\sin\varphi-y\cos\varphi \\
	0 & 0 & 1
	\end{array}\right)}$$
$$=w_1e_1+(w_1y+ w_2\cos\varphi)e_2+(-w_1x+ w_2\sin\varphi)e_3,$$
$$\psi_g(w)=w_1\cos\alpha+(w_1y+w_2\cos\varphi)\sin\alpha+(-w_1x+ w_2\sin\varphi)\beta=$$
$$w_1(\cos\alpha+  y\sin\alpha- \beta x) + w_2(\sin\alpha\cos\varphi + \beta\sin\varphi).$$
Consequently,
$$u(g)=(\cos\alpha + y\sin\alpha-\beta x)e_1+(\sin\alpha\cos\varphi+\beta\sin\varphi)e_2,$$
$$v(g)=gu(g)=\tiny{\left(\begin{array}{ccc}
	\cos\varphi & -\sin\varphi & x\\
	\sin\varphi & \cos\varphi & y \\
	0 & 0 & 1
	\end{array}\right)\left(\begin{array}{ccc}
	0 & -\cos\alpha-y\sin\alpha+\beta x & \sin\alpha\cos\varphi+\beta\sin\varphi\\
	\cos\alpha+y\sin\alpha-\beta x & 0 & 0 \\
	0 & 0 & 0
	\end{array}\right)}$$
$$=\tiny{\left(\begin{array}{ccc}
	\sin\varphi(\beta x-\cos\alpha-y\sin\alpha) & \cos\varphi(\beta x-\cos\alpha-y\sin\alpha) & \cos\varphi(\sin\alpha\cos\varphi+\beta\sin\varphi) \\
	\cos\varphi(\cos\alpha+y\sin\alpha-\beta x) & \sin\varphi(\beta x-\cos\alpha-y\sin\alpha) & 
	\sin\varphi(\sin\alpha\cos\varphi+\beta\sin\varphi) \\
	0 & 0 & 0
	\end{array}\right)}.$$
Hence integral curves of vector field $v(g),$ $g\in SE(2),$ satisfy the system of differential equations
\begin{equation}
\label{sisteme}
\left\{\begin{array}{l}
\dot{\varphi}=\cos\alpha+y\sin\alpha-\beta x, \\
\dot{x}=\cos\varphi(\sin\alpha\cos\varphi+\beta\sin\varphi), \\
\dot{y}=\sin\varphi(\sin\alpha\cos\varphi+\beta\sin\varphi) \\
\end{array}\right.
\end{equation}
The geodesic $g(t),$ $t\in\mathbb{R}$, with $g(0)=e$ is a solution of this system with initial data $\varphi(0)=x(0)=y(0)=0$. In this case, $|u(g(t))|\equiv 1$, i.e.
\begin{equation}
\label{m1}
g(t)\in M_1=\{(\sin\alpha\cos\varphi+\beta\sin\varphi)^2+(\cos\alpha+y\sin\alpha-\beta x)^2=1\}\subset SE(2).
\end{equation}
Therefore there exist differentiable functions $\omega=\omega(t)=2\xi(t)$ such that
\begin{equation}
\label{s00}
\sin\frac{\omega(t)}{2}=\sin\alpha\cos\varphi+\beta\sin\varphi,\quad\cos\frac{\omega(t)}{2}=\cos\alpha+y\sin\alpha-\beta x.
\end{equation}
Given the equality $\varphi(0)=x(0)=y(0)=0$, we can assume that $\omega(0)=2\xi(0)=2\alpha$.

On the ground of formula (\ref{s00}) the system (\ref{sisteme}) is written in a form
\begin{equation}
\label{sisteme1}
\left\{\begin{array}{l}
\dot{\varphi}=\cos{\frac{\omega}{2}}, \\
\dot{x}=\sin{\frac{\omega}{2}}\cos\varphi, \\
\dot{y}= \sin{\frac{\omega}{2}}\sin\varphi.
\end{array}\right.
\end{equation}

Differentiating the first and the second equalities in (\ref{s00}) and using (\ref{sisteme1}), we get
$$\frac{\dot{\omega}}{2}\cos\frac{\omega}{2}=-\left(\sin\alpha\sin\varphi-\beta\cos\varphi\right)\dot{\varphi}= 
-\cos\frac{\omega}{2}\left(\sin\alpha\sin\varphi-\beta\cos\varphi\right),$$
$$-\frac{\dot{\omega}}{2}\sin\frac{\omega}{2}=\dot{y}\sin\alpha-\beta\dot{x}=
\sin{\frac{\omega}{2}}\left(\sin\alpha\sin\varphi-\beta\cos\varphi\right),$$
whence
\begin{equation}
\label{init}
\dot{\omega}=2(\beta\cos\varphi-\sin\xi\sin\varphi),\quad \dot{\omega}(0)=2\dot{\xi}(0)=2\beta.
\end{equation}
Differentiating the last equality, we get in view of formulae (\ref{s00}) and (\ref{sisteme1})
\begin{equation}
\label{pendulum}
\ddot{\omega}=-2(\beta\sin\varphi+\sin\alpha\cos\varphi)\dot{\varphi}=-2\sin\frac{\omega}{2}\cos\frac{\omega}{2}=
-\sin\omega.
\end{equation}

We get again  the mathematical pendulum equation.

\end{document}